\documentclass[12pt]{article}
\usepackage[centertags]{amsmath}
\usepackage{amsfonts}
\usepackage{amssymb}
\usepackage{amsthm}
\input{amssym.def}

\numberwithin{equation}{section}

\theoremstyle{plain}
\newtheorem{Definition}{Definition}[section]
\newtheorem{definition}[Definition]{Definition}
\newtheorem{theorem}[Definition]{Theorem}
\newtheorem{lemma}[Definition]{Lemma}

\newtheorem{corollary}[Definition]{Corollary}
\newtheorem{example}[Definition]{Example}

\newcommand{\mbz}{\mathbb{Z}}

\newcommand{\mbr}{\mathbb{R}}

\newcommand{\ms}{\mathcal{M}(S)}

\begin{document}
\title{\Large \bf On the Jacobson radical and semisimplicity of a semiring}
\author{A. K. Bhuniya and Puja Sarkar}
\date{}
\maketitle

\begin{center}
Department of Mathematics, Visva-Bharati, \\ Santiniketan-731235, India. \\
anjankbhuniya@gmail.com; puja.vb.math@gmail.com
\end{center}

\begin{abstract}
Based on the minimal and simple representations, we introduce two Jacobson-type Hoehnke radicals, m-radical and s-radical, of a semiring $S$. Every minimal (simple) $S$-semimodule is a quotient of $S$ by a regular right congruence (maximal) $\mu$ on $S$ such that $[0]_\mu$ is a maximal $\mu$-saturated right ideal in $S$. Thus the m(s)-radical becomes an intersection of some regular congruences. Finally, every semisimple semiring is characterized as a subdirect product of primitive semirings; and every s-primitive semiring is represented as a 1-fold transitive subsemiring of the semiring of all endomorphisms on a semimodule over a division semiring.
\end{abstract}

\section{Introduction}
A semiring is an algebraic structure satisfying all the axioms of a ring, but one that every element has an additive inverse. The absence of additive inverses forces a semiring to deviate radically from behaving like a ring. For example, ideals are not in bijection with the congruences on a semiring. Further, the presence of the additively idempotent semirings makes the class of the semirings aberrant. Now semirings have become a part of mainstream mathematics for their importance in theoretical computer science \cite{Schutzenberger} and automata theory \cite{Conway}; and for the surprising `characteristic one analogy' of the usual algebra over fields \cite{CC2009, CC2010, CCM, Deitmar2008}; and for the role of additively idempotent semirings in tropical mathematics \cite{ABG, Castella2010, Gathmann, IMS, IR2010, Litviniv05-07}.

There are many papers on the concrete radicals \cite{Bourne1951, HW1997, HW2001, Ilzuka1959, KN2014, Latorre1967,  MT} as well as on abstract theory of radicals \cite{Morak, OJ1983, ON1989, WW1992, WW2003a, WW2003b}  on semirings. Bourne \cite{Bourne1951} characterized Jacobson radical of a semiring $S$ internally as the sum of all right semiregular ideals in $S$. Ilzuka \cite{Ilzuka1959} defined the same in terms of the irreducible representations of a semiring. Also, he introduced quasi-regularity in semirings and characterized the Jacobson radical of a semiring $S$ as the intersection of all strongly closed primitive ideals in $S$. If $S$ is an additively idempotent semiring, then every right ideal in $S$ is semiregular and quasi-regular. Thus according to both Bourne and Ilzuka, every additively idempotent semiring $S$ is a radical semiring, i.e., $J(S)=S$. In a recent paper, \cite{KN2014} Katsov and Nam introduced and studied an external Kurosh-Amitsur Jacobson radical theory for a semiring $S$ based on representations of $S$. In search of a suitable analogue of the Jacobson radical that will work for additively idempotent semirings also, they introduced another radical $J_s(S)$ of a semiring $S$ in terms of the simple representations of $S$ and characterized finite additively idempotent $J_s$-semisimple semirings. Mai and Tuyen \cite{MT} continued to study this radical $J_s(S)$ and $J_s$-semisimplicity within the class of zerosumfree semirings.

Thus even though the ideals are the fundamental objects of the radical theory of rings, they can not fulfil the same role in the radical theory of semirings. At the same time, replacing ideals with the more general notion of congruences on semirings exhibits many excellent properties and several analogies with classical results on the rings \cite{BE2017, JM2018JA, JM2018SM}. Here we define and study two Jacobson type radicals of a semiring $S$ as congruences on $S$, and also show that these radicals act for additively idempotent semirings.

In our approach, the annihilator of an $S$-semimodule $M$ is considered as a congruence on $S$. Similarly to the semirings, subsemimodules are not in bijection with the congruences on a semimodule; which produces three variants of `irreducibility' of semimodules -- minimal semimodules, elementary semimodules, and simple semimodules \cite{Chen}, \cite{IRS2011}. Based on the classes of minimal semimodules and simple semimodules, two different Jacobson type Hoehnke radicals of a semiring $S$ are defined in terms of annihilator of $S$-semimodules, which are called m-radical and s-radical of $S$ respectively. The absence of additive inverse makes it difficult to define the regularity of an ideal in a semiring analogously to the rings, whereas in the case of a congruence, it can be defined naturally. An $S$-semimodule $M$ is (simple) minimal if and only if $M \simeq S/ \mu$ where $\mu$ is a (maximal) regular right congruence on $S$ such that $[0]_\mu$ is a maximal $\mu$-saturated right ideal in $S$. Thus to make the quotient semimodule $S/ \mu$ `irreducible', the maximality of the regular right congruence $\mu$ is not sufficient; indeed, it is shared with its zero class $[0]_\mu$. Such a sharing among the right congruences and the right ideals happens in most of the theorems of this article. Also, considering radical as a congruence makes it possible to present the $J$-semisimple semirings as a subdirect product of primitive semirings.

This paper is organized as follows. Section 2 briefly recaps necessary definitions and associated facts on semirings and semimodules. Section 3  introduces the Jacobson m-radical and s-radical externally. Each of these radicals is Hoehnke radical; and can be expressed as an intersection of a suitable class of regular right congruences. We conclude this section with the characterization of the radicals of the product semirings. Section 4 introduces m-primitive and s-primitive semirings and characterizes Jacobson semisimple semirings as a subdirect product of primitive semirings. Every commutative (s)m-primitive semiring is a (congruence simple) semifield. Since every congruence simple semifield $S$ with $|S| > 2$ is a field, every commutative s-semisimple semiring is a subdirect product of a family of semirings, each of which is either the 2-element Boolean algebra or a field. Finally, every s-primitive semiring is represented as a 1-fold transitive subsemiring of the semiring of all endomorphisms on a semimodule over a division semiring.

\section{Preliminaries}
In this paper by a \emph{semiring} $(S, +, \cdot)$ we mean a nonempty set $S$ with two binary operations `+' and `$\cdot$' satisfying:
\begin{itemize}
\item
$(S, +)$ is a commutative monoid with identity element 0;
\item
$(S, \cdot)$ is a semigroup;
\item
$a(b+c)=ab+ac$ and $(a+b)c=ac+bc$ for all $a, b, c \in S$.
\end{itemize}
Following Golan \cite{Golan1999}, in the literature, it is generally assumed that a semiring has both an additive identity and a multiplicative identity and the additive identity 0 is absorbing, i.e., $0s=s0=0$ for all $s \in S$. However, we follow the convention of Hebisch and Weinert \cite{HW}. Also, there are many articles on semirings where existence of unity is not assumed \cite{BM2010}, \cite{BM2015}, \cite{Il'in2021}, \cite{IRS2011}, \cite{SB2011}, \cite{WJK}. In this paper, we only assume that every semiring has an additive identity element 0 which is absorbing. If a semiring $S$ with multiplicative identity is such that every nonzero element has a multiplicative inverse, then $S$ is called a \emph{division semiring}. A commutative division semiring is called a \emph{semifield}. A semiring $S$ is said to be an \emph{additively idempotent semiring} if $a + a =a$ for all $a \in S$. Both the two element Boolean algebra $\mathbb{B}$ and the max-plus algebra $\mathbb{R}_{max}$ are additively idempotent semifields.

A \emph{right congruence} $\rho$ on a semiring $S$ is an equivalence relation on $S$ such that for all $a, b, c \in S$, if $(a, b) \in \rho$ then we have $(a+c, b+c) \in \rho$ and $(ac, bc) \in \rho$. A \emph{left congruence} on a semiring is defined similarly, and a \emph{congruence} is both a left and a right congruence. Similar to the rings, the ideals and homomorphisms of semirings are defined in the usual way. Also, it is assumed that every semiring homomorphism $\phi: S_1 \longrightarrow S_2$ satisfies $\phi(0_1)=0_2$. A bijective homomorphism is called an \emph{isomorphism}; if there is an isomorphism $\phi: S_1 \longrightarrow S_2$, then the semirings $S_1$ and $S_2$ are said to be \emph{isomorphic} which is denoted by $S_1 \simeq S_2$. The kernel of a semiring homomorphism $\phi : S_1 \longrightarrow S_2$ is defined by $\ker \phi = \{(a, b) \in S_1 \times S_2 \mid \phi(a)=\phi(b)\}$. Then $\ker \phi$ is a congruence on $S_1$ and $S_1/\ker \phi \simeq \phi(S_1)$.

We denote $\Delta_S=\{(s, s) \mid s \in S\}$ and $\nabla_S=S \times S$. A semiring $S$ is said to be \emph{congruence-simple} if it has no congruences other than $\Delta_S$ and $\nabla_S$. If $\rho$ is a (left, right) congruence on $S$, then for every $s \in S$, the $\rho$-class containing $s$ is denoted by $[s]_\rho$ or briefly by $[s]$ if there is no scope of ambiguity.

For every (left, right) ideal $I$ in $S$, the Bourne (left, right) congruence $\sigma_I$ on $S$ is defined by: for $x, y \in S$, $x \sigma_I y$ if $x + i_1 = y + i_2$ for some $i_1, i_2 \in I$. The Bourne (left, right) congruence $\sigma_I$ is the smallest (left, right) congruence on $S$ such that $I \subseteq [0]_{\sigma_I}$. In general $[0]_{\sigma_I} \neq I$. In fact, $[0]_{\sigma_I} = \{ x \in S \ | \ x + i = i \; \textrm{for some} \; i \in I\}$ which is denoted by $\overline{I}$ and is called the \emph{saturation of $I$}. A (left, right) ideal $I$ is called \emph{saturated} if $\overline{I} = I$ \cite{Lescot-II}.

A \emph{right $S$-semimodule} is a commutative monoid $(M, +, 0_{M})$ equipped with a right action $M \times S \longrightarrow M$ that satisfies for all $m, m_1, m_2 \in M$ and $s, s_1, s_2 \in S$:
\begin{itemize}
\item
$(m_1+m_2)r=m_1r+m_2r$;
\item
$m(r_1+r_2)=mr_1+mr_2$;
\item
$m(r_1r_2)=(mr_1)r_2$;
\item
$m0=0_M=0_Mr$.
\end{itemize}
Unless stated otherwise, by an \emph{$S$-semimodule} $M$, we mean a right $S$-semimodule.

Subsemimodules and congruences on a semimodule are defined in the usual way. The set of all subsemimodules and congruences on an $S$-semimodule $M$ will be denoted by $\mathcal{S}_S(M)$ and $\mathcal{C}_S(M)$ respectively.

Following Chen et. al. \cite{Chen} we define:
\begin{definition}
Let $M$ be an $S$-semimodule such that $MS \neq 0 $. Then $M$ is called,
\begin{enumerate}
\item[(i)]
minimal if $M$ has no subsemimodules other than $(0)$and $M$ ;
\item[(ii)]
simple if it is minimal and the only congruences on $M$ are $\Delta_M$ and $\nabla_M$ where $\Delta_M$ is the equality relation on $M$ and $\nabla_M=M \times M$.
\end{enumerate}
\end{definition}
In \cite{IRS2011}, simple semimodules have been termed as irreducible semimodules. We denote the class of all minimal and simple $S$-semimodules by $\mathcal{M}(S)$ and  $\mathcal{S}(S)$, respectively.

It was remarked in \cite{IRS2011} that in a minimal semimodule $M$, every $m \neq 0$ generates $M$. However, for the sake of completion, here we include proof.
\begin{lemma}     \label{minimal}
A nonzero $S$-semimodule $M$ is minimal if and only if $M = mS$ for all $m (\neq 0) \in M$.
\end{lemma}
\begin{proof}
Let $M$ be minimal. Then $MS \neq 0$ implies that $N = \{ m \in M \mid mS = 0\}$ is a proper subsemimodule of $M$, and hence it must be equal to 0. So for every $m \neq 0$, $mS$ is a nonzero subsemimodule of $M$, and it follows that $mS = M$.

The converse is trivial.
\end{proof}

Let $R$ be a ring and $M$ be a minimal semimodule over $R$. If $m (\neq 0) \in M$, then there exists $s \in S$ such that $m=ms$. Consider $m(-s) \in M$. Then $ms+m(-s)=m(s-s)=m0=0$ implies that every element of $M$ has an additive inverse. Hence $(M, +)$ is a group, and so $M$ is a module over $R$. Thus over a ring $R$, minimal semimodules and simple semimodules coincide.

The \emph{annihilator} of an $S$-semimodule $M$ is defined by
\[
ann_S(M) = \{ (s_1, s_2) \in S \times S \ | \ ms_1 = ms_2 \; \textrm{for all} \; m \in M\}.
\]
Then $ann_S(M)$ is a congruence on $S$.

The right action of $S$ on $M$ induces an $S$-endomorphism $\psi_s : M \longrightarrow M$ ; $m \mapsto ms$ on $M$ for each $s \in S$. Denote the semiring of all $S$-endomorphisms on $M$ by $End_S(M)$. Thus we get a representation $\psi : S \longrightarrow End_S(M)$; $s \mapsto \psi_s$ of $S$. In this case, $M$ is called a representation module of $S$. Note that $\ker \psi = \{ (s_1, s_2) \in S \times S \ | \ \psi(s_1) = \psi(s_2)\} = ann_S(M)$.

An $S$-semimodule $M$ is said to be \emph{faithful} if $ann_S(M) = \Delta_S$, equivalently $\ker \psi = \Delta_S$.

Let $\rho$ be a right congruence on $S$. Define $S/ \rho \times S \longrightarrow S/\rho$ by $([a]_\rho, s) \mapsto [as]_\rho$. Then $S/\rho$ is a right $S$-semimodule. Also, if $M$ is a right $S$-semimodule then for every congruence $\rho$ on $S$ with $\rho \subseteq ann_S(M)$, the scalar multiplication $m[s]_\rho = ms$ makes $M$ an $S/\rho$-semimodule.

In the sequel, we will have several occasions to use the following result, which can be proved easily, and so we omit the proof.
\begin{lemma}                \label{Lemma 2.1}
Let $S$ be a semiring and $\rho$ be a congruence on $S$.
\begin{enumerate}
\item
If $M$ is an $S/\rho$-semimodule, then $M$ becomes an $S$-semimodule under the scalar multiplication $ms =
m [s]$.

Moreover, $\rho \subseteq ann_S(M)$.
\item
Let $M$ be an $S$-semimodule and $\rho \subseteq ann_S(M)$. Then
  \begin{enumerate}
  \item
  $ann_{S/\rho}(M) = ann_S(M)/\rho$.
  \item
  $\mathcal{S}_S(M) = \mathcal{S}_{S/\rho}(M)$.
  \item
  $\mathcal{C}_S(M) = \mathcal{C}_{S/\rho}(M)$.
  \end{enumerate}
\end{enumerate}
\end{lemma}

The reader is referred to \cite{Golan1999} for the undefined terms and notions concerning semirings and semimodules over semirings.

\section{Jacobson radical of a semiring}
In this section, we define Jacobson m-radical and s-radical of a semiring based on the classes of minimal semimodules and simple semimodules, respectively, in a way that is familiar from the radical theory of rings. Also, we find two suitable classes of regular right congruences on a semiring $S$ to characterize these two radicals internally without any reference to the semimodules over $S$.
\begin{definition}
Let $S$ be a semiring. We define
\begin{enumerate}
\item[(i)]
$m$-radical of $S$ by $rad_m(S) = \cap_{M \in \mathcal{M}(S)} ann_S(M)$;
\item[(ii)]
$s$-radical of $S$ by $rad_s(S) = \cap_{M \in \mathcal{S}(S)} ann_S(M)$.
\end{enumerate}

If there are no minimal semimodules over $S$, then we define $rad_m(S)=\nabla_S$. Similarly, we define $rad_s(S)=\nabla_S$ if there is no simple $S$-semimodules.

A semiring $S$ is said to be $m$-semisimple if $rad_m(S)=\Delta_S$; and $s$-semisimple if $rad_m(S)=\Delta_S$.
\end{definition}

Recall that a right module $M$ over a ring $R$ is called irreducible if $MR \neq \{0\}$ and $M$ has no submodules other than $\{0\}$ and $M$. The Jacobson radical of a ring $R$ is defined by $J(R)=\cap Ann_R(M)$ where the intersection runs over all irreducible $R$-modules and $Ann_R(M)=\{r \in R \mid mr=0 \ \mbox{for all} \ m \in M\}$ \cite{Herstein}. Since every ideal in a ring is saturated, it follows that $ann_R(M)=\sigma_{Ann_R(M)}$ and $[0]_{ann_R(M)}=Ann_R(M)$ for every right $R$-module $M$.
\begin{example}
Let $R$ be a ring (possibly without 1). If $M$ is a minimal semimodule over $R$, then, by Lemma \ref{minimal}, $(M, +)$ is a group, and so is a module over $R$. An $R$-semimodule $M$ is minimal if and only if it is simple; equivalently, $M$ is an irreducible module over $R$. Hence it follows that $rad_m(R)=rad_s(R)$. Also we have $rad_s(R)=\cap_{M \in \mathcal{S}(R)}ann_R(M)=\cap_{M \in \mathcal{S}(R)}\sigma_{Ann_R(M)} = \sigma_{\cap_{M \in \mathcal{S}(R)}Ann_R(M)}=\sigma_{J(R)}$  and $[0]_{rad_s(R)}=J(R)$.
\end{example}

An assignment from the collection of all semirings to the collection of all congruences over semirings $S \longmapsto r(S)$ is said to be a \emph{Hoehnke radical} if for every onto homomorphism $f: S \rightarrow f(S)$,
\begin{enumerate}
\item[(i)]
$f(r(S)) \subseteq r(f(S))$ where $f(r(S)) = \{ (f(a) , f(b))\mid (a,b) \in r(S) \}$;
\item[(ii)]
$r(S/r(S)) = \Delta_{(S/r(S))}$.
\end{enumerate}

We show that both the m-radical and the s-radical are Hoehnke radicals on $S$.
\begin{theorem}
Let $S$ be a semiring. Then both the assignments $S \longmapsto rad_m(S)$ and $S \longmapsto rad_s(S)$ are Hoehnke radicals on $S$.
\end{theorem}
\begin{proof}
We prove the result for the m-radical and the proof for the s-radical is similar. Let $f: S \rightarrow f(S)$ be an onto homomorphism. Then the first isomorphism theorem for semirings implies $S/\ker f \cong f(S)$.

Let $M$ be a minimal $f(S)$-semimodule. Then, by Lemma \ref{Lemma 2.1}, $M$ becomes a minimal $S$-semimodule under the scalar multiplication $ms = mf(s)$. Hence it follows that
\begin{align*}
f(rad_m(S)) &= \{ (f(a) , f(b)) \mid (a , b) \in rad_m(S) \}\\
            &= \{ (f(a) , f(b)) \mid (a , b) \in ann_S(M)\mbox{ for all } M \in \mathcal{M}(S)\}\\
            & \subseteq \{ (f(a) , f(b)) \mid (a , b) \in ann_S(M)\mbox{ for all } M \in \mathcal{M}(f(S)) \} \\
            &= \{ (f(a) , f(b)) \mid (f(a) , f(b)) \in ann_{f(S)}(M)\mbox{ for all } M \in \mathcal{M}(f(S)) \}\\
            &= rad_m(f(S)).
\end{align*}
Also, by Lemma \ref{Lemma 2.1}, we have
\begin{align*}
rad_m( S/rad_m(S) ) &= \cap_{M \in \mathcal{M}(S/rad_m(S))} ann_{S/rad_m(S)}(M) \\
                    &= \cap_{M \in \mathcal{M}(S/rad_m(S))} ann_S (M)/rad_m(S)  \\
                    &= \cap_{M \in \mathcal{M}(S)} ann_S (M) /rad_m(S) \\
                    &= \Delta_{S/rad_m(S)}.
\end{align*}
Therefore the assignment $rad_m(S)$ is a Hoehnke radical.
\end{proof}
\begin{example}
Let $G$ be a finite group and $\mathbb{B}G$ its group semiring over the two-element Boolean algebra $\mathbb{B}$. Let $M$ be a minimal $\mathbb{B}G$-semimodule. Then, by Theorem 3.3 \cite{Chen}, $M$ is isomorphic to the trivial semimodule $\mathbb{B}$. Thus $ann_{\mathbb{B}G}(M) = \nabla_{\mathbb{B}G}$. Therefore $rad_m(\mathbb{B}G) = \nabla_{\mathbb{B}G}$.

Also, a $\mathbb{B}G$-semimodule $M$ is minimal if and only if it is simple \cite{Chen}. Hence $rad_s(\mathbb{B}G) = rad_m(\mathbb{B}G) = \nabla_{\mathbb{B}G}$.
\end{example}

A right congruence $\mu$ on $S$ is said to be a \emph{regular right congruence} if there exists $e \in S$ such that $(es, s) \in \mu$ for every $s \in S$.

If $\rho$ is a regular right congruence on $S$, then $M=S/\rho$ is a right $S$-semimodule such that $MS \neq 0$. Suppose that $N$ is a subsemimodule of $M$. Then $I(N)=\{s \in S \mid [s]_\rho \in N\}$ is a right ideal in $S$ that satisfies the property: for every $s \in S$ and $i \in I$, $(s, i) \in \rho$ implies that $s \in I$. Conversely, for a right ideal $I$ of $S$ that satisfies the above property induces a subsemimodule $N(I)=\{[s]_\rho \mid s \in I\}$ of the right $S$-semimodule $S/\rho$.

Let $I$ be a (left, right) ideal of a semiring $S$ and $\mu$ be a (left, right) congruence relation on $S$. Then $I$ is said to be a \emph{$\mu$-saturated (left, right) ideal} of $S$ if for every $s \in S$ and $i \in I$, $(s, i) \in \mu$ implies that $s \in I$.

An ideal $I$ is $\mu$-saturated if and only if $I=\cup_{a \in I}[a]_\mu$. Also, $I$ is a saturated ideal if and only if it is $\sigma_I$-saturated. Thus the $\mu$-saturated ideals generalize the notion of saturated ideals.

The subsequent two theorems characterize the regular congruences $\mu$ on $S$ such that the quotient semimodule $S/ \mu$ is a minimal or a simple semimodule over $S$.

\begin{theorem}    \label{iff minimal}
Let $M$ be an $S$-semimodule. Then $M$ is minimal if and only if there exists a regular right congruence $\mu$ on $S$ such that $S/\mu \simeq M$ and $[0]_{\mu}$ is a maximal $\mu$-saturated right ideal in $S$.
\end{theorem}
\begin{proof}
Let $M$ be minimal. Consider $m\in M, m\neq0$. Then, by Lemma \ref{minimal}, we have $mS = M$. Define $\phi : S \rightarrow M$ by $\phi(s) = ms$ for all $s\in S$. Then $\phi$ is an onto module homomorphism. Hence $\mu = \ker \phi = \{(s_1, s_2) \in S \times S \mid \; ms_1 = ms_2\}$ is a right congruence on $S$; and we have $M \simeq S/\mu$ as $S$-semimodules. Since $mS = M$, there exists an element $e \in S$ such that $me = m$, which implies that $mes = ms$ for all $s \in S$. Therefore $(es,s)\in\mu \mbox{ for all }s \in S$. Hence, $\mu$ is a regular right congruence on $S$. Let $I$ be a $\mu$-saturated right ideal in $S$ such that $[0]_{\mu} \subsetneq I$. Then $J = \{ [s]_{\mu} \mid \; s\in I\}$ is a subsemimodule of the $S$-semimodule $S/\mu$. Now $S/\mu$ is minimal and $J \neq \{[0]_{\mu}\}$ implies that $J=S/\mu$. Since $I$ is $\mu$-saturated, it follows that $I=S$. Thus $[0]_{\mu}$ is a maximal $\mu$-saturated right ideal.

Conversely assume that $\mu$ is a regular right congruence on $S$ such that $[0]_{\mu}$ is a maximal $\mu$-saturated right ideal in $S$. Then $S/\mu$ is a right $S$-semimodule. Since $\mu$ is regular, we have an element $e \in S$ such that $[e]_{\mu}s = [s]_{\mu}$ for all $s \in S$. Hence $[e]_{\mu}S = S/\mu$ and so $(S/\mu)S = S/\mu$. Now take any nonzero subsemimodule $N$ of $S/\mu$. Then $I(N) = \{ s \in S \mid \; [s]_{\mu} \in N\}$ is a $\mu$-saturated right ideal in $S$ containing $[0]_{\mu}$. Since $I(N) \neq [0]_{\mu}$, we have $I(N) = S$. Thus $N = S/\mu$ implies that $S/\mu$ is minimal.
\end{proof}

Every simple semimodule is congruence-simple. Therefore for every right congruence $\mu$ on $S$, $S/\mu$ is simple, implying that $\mu$ is maximal. The following theorem characterizes a regular right congruence $\mu$ on $S$ that makes the right $S$-semimodule $S/\mu$ simple.
\begin{theorem}              \label{iff simple}
Let $M$ be a $S$-semimodule. Then $M$ is simple if and only if there exists a maximal regular right congruence $\mu$ on $S$ such that $S/\mu \simeq M$ and $[0]_{\mu}$ is a maximal $\mu$-saturated right ideal in $S$.
\end{theorem}
\begin{proof}
Let $M$ be simple. Then, by Theorem \ref{iff minimal}, there exists a regular right congruence $\mu$ on $S$ such
that $S/\mu \simeq M$ and $[0]_{\mu}$ is a maximal $\mu$-saturated right ideal in $S$. Now for any regular right congruence $\phi$ on $S$ containing $\mu$ we have a right congruence $\phi/\mu = \{([a]_{\mu},[b]_{\mu}) \mid \; (a,b) \in \phi\}$ on the $S$-semimodule $S/\mu$. Since $S/\mu$ is simple, $\phi/\mu$ is either $\Delta_{S/\mu}$ or $\nabla_{S/\mu}$. Therefore $\phi$ is either $\mu$ or $\nabla_S$. Thus $\mu$ is a maximal regular right congruence on $S$.

Conversely, by Theorem \ref{iff minimal}, it follows that $S/\mu$ is a minimal $S$-semimodule for every maximal
regular right congruence $\mu$ on $S$ where $[0]_{\mu}$ is a maximal $\mu$-saturated right ideal in $S$. Since every maximal regular right congruence is also a maximal right congruence, $S/\mu$ is simple.
\end{proof}

Theorem \ref{iff minimal} and Theorem \ref{iff simple} motivate two further definitions. A regular right congruence $\mu$ on $S$ is said to be \emph{m-regular} if $[0]_{\mu}$ is a maximal $\mu$-saturated right ideal in $S$ and said to be \emph{s-regular} if it is a maximal regular right congruence such that $[0]_{\mu}$ is a maximal $\mu$-saturated right ideal in $S$. We denote the set of all m-regular right congruences on $S$ by $\mathcal{RC}_m(S)$ and the set of all s-regular right congruences on $S$ by $\mathcal{RC}_s(S)$.

Now m-radical and s-radical of a semiring $S$ are characterized in terms of the m-regular and the s-regular right congruences, respectively. First, we must associate the regular right congruences on a semiring $S$ with the semimodules over $S$. Let $M$ be a right $S$-semimodule. For every $m \in M$, we define
\[
\delta_m = \{(a,b) \in S\times S \mid \; ma = mb\}.
\]
Then $\delta_m$ is a right congruence on $S$. The following result shows that every regular right congruence is of this form. The following result is by analogy with \cite{Hoehnke1964}.
\begin{lemma}                          \label{delta m}
Let $S$ be semiring and $M$ be a right $S$-semimodule. Then
\begin{enumerate}
\item[(i)]
$ann_S(M) = \cap_{m \in M}\delta_m$.
\item[(ii)]
for every regular right congruence $\mu$ on $S$, there exists an element $e \in S$ such that $\delta_{[e]_{\mu}} = \mu$ where $[e]_\mu$ is an element of the right $S$-semimodule $S/\mu$.
\item[(iii)]
if moreover, $M$ is minimal then for each $m (\neq 0) \in M$, $\delta_m$ is an m-regular right congruence on $S$.
\item[(iv)]
if moreover, $M$ is simple then for each $m (\neq 0) \in M$, $\delta_m$ is an s-regular right congruence on $S$.
\end{enumerate}
\end{lemma}
\begin{proof}
(i) Follows trivially. \\
(ii) \ Since $\mu$ is a regular right congruence on $S$, there exists an element $e \in S$ such that for every $s \in S$, $(es,s) \in \mu$ and so $[e]_{\mu}s = [s]_{\mu}$ in the right $S$-semimodule $S/\mu$.

Hence for every, $s,t \in S$, $(s,t) \in \mu, \Leftrightarrow [s]_{\mu} = [t]_{\mu} \Leftrightarrow [e]_{\mu}s = [e]_{\mu}t \Leftrightarrow (s,t) \in \delta_{[e]_{\mu}}$ which implies that $\mu = \delta_{[e]_{\mu}}$.
\\ (iii) \ Let $M$ be a minimal $S$-semimodule and $m(\neq 0) \in M$. Then $mS = M$, by Lemma \ref{minimal}. Hence there exists an element $a \in S$ such that $ma = m$. Then for every $s \in S$, we have $mas = ms $, which implies that $(as, s) \in \delta_m$. Thus $\delta_m$ is a regular right congruence where $[0]_{\delta_m} = \{ s \in S \mid ms = 0 \}$. Now $s \delta_m t $ and $t \in [0]_{\delta_m}$ implies that $ms = mt = 0$ and so $s \in [0]_{\delta_m}$. Thus $[0]_{\delta_m}$ is a $\delta_m$-saturated right ideal in $S$.

Let $I$ be a $\delta_m$-saturated right ideal in $S$ such that $[0]_{\delta_m} \subsetneq I$. Consider an element $x \in I \backslash [0]_{\delta_m}$. Then $mx \neq 0$, which implies that $(mx)S = M$. So for each $s \in S$, there exists an element $t \in S$ such that $mxt = ms$, which implies that $(xt,s) \in \delta_m$. Since $xt \in I$, it follows that $s \in I$. Hence $S = I$ and so $[0]_{\delta_m}$ is a maximal $\delta_m$-saturated ideal in $S$.
\\ (iv) \ Let $M$ be a simple $S$-semimodule and $m(\neq 0) \in M$. Then by (3), we have $\delta_m$ is a regular right congruence on $S$ such that $[0]_{\delta_m}$ is a maximal $\delta_m$-saturated ideal in $S$. Let $\phi$ be a right congruence on $S$ containing $\delta_m$. We define $\phi_M = \{ (ms,mt) \in M \times M \mid (s,t) \in \phi \}$. Then for every $m \neq 0$ in $M$, $mS = M$, implies that $\phi_M$ is reflexive. Also, it follows from the definition that $\phi_M$ is a congruence on $M$. Hence $\phi_M$ is either $\Delta_M$ or $\nabla_M$ which implies that $\phi$ is either $\delta_m$ or $\nabla_S$. Therefore $\delta_m$ is an s-regular right congruence on $S$.
\end{proof}

If 0 is the zero element of an $S$-semimodule $M$, then $\delta_0=S \times S$, and so $ann_S(M) = \cap_{m \neq 0}\delta_m$. Thus Lemma \ref{delta m} tells us that the annihilator of every minimal (simple) $S$-semimodule $M$ can be expressed as the intersection of the family $\{\delta_m \mid m (\neq 0) \in M\}$ of m-regular (s-regular) right congruences on $S$. This characterization of the annihilators of the minimal and simple semimodules gives the following characterization of the m-radical and s-radical of a semiring.
\begin{theorem}
Let $S$ be a semiring. Then
\begin{enumerate}
\item[(i)]
$rad_m(S) = \cap_{\mu \in \mathcal{RC}_m(S)} \mu$
\item[(ii)]
$rad_s(S) = \cap_{\mu \in \mathcal{RC}_s(S)} \mu$
\end{enumerate}
\end{theorem}
\begin{proof}
(i) \ If $M$ is a minimal $S$-semimodule then for all $m \in M, m\neq 0$, $\delta_m \in \mathcal{RC}_m(S)$ by  Lemma \ref{delta m}. Hence $\cap_{\mu \in \mathcal{RC}_m(S)}\mu \subseteq \cap_{m(\neq 0)\in M}\delta_m = ann_S(M)$. Thus it follows that $\cap_{\mu \in \mathcal{RC}_m(S)}\mu \subseteq \cap_{M\in \ms}ann_S(M)$.

Now let $(a,b) \in \cap_{M \in \ms}ann_S(M)$. Then Theorem \ref{iff minimal} implies that $(a,b) \in ann_S(S/\mu)$ for every $\mu \in \mathcal{RC}_m(S)$. For each $\mu \in \mathcal{RC}_m(S)$, we have $ann_S(S/\mu) = \cap_{s\in S}\delta_{[s]_{\mu}}$. Also there exists an element $e\in S$ such that $\delta_{[e]_{\mu}} = \mu$ which implies that $(a,b) \in \mu$. Hence $(a,b) \in \cap_{\mu\in \mathcal{RC}_m(S)} \mu$ and it follows that $ann_S(M) \subseteq \cap_{\mu\in \mathcal{RC}_m(S)}\mu$. Thus we have $rad_m(S) = \cap_{\mathcal{M}(S)}ann_S(M) = \cap_{\mu \in \mathcal{RC}_m(S)}\mu$.
\\ (ii) \ If $M$ is a simple $S$-semimodule, then for all $m \in M (m \neq 0)$, $\delta_m \in \mathcal{RC}_s(S)$ by Lemma \ref{delta m}. Thus $\cap_{\mu \in \mathcal{RC}_s(S)} \mu \subseteq \cap_{m(\neq 0) \in M} \delta_m = ann_S(M)$. Therefore $\cap_{\mu \in \mathcal{RC}_s(S)} \mu \subseteq \cap_{M \in \mathcal{S}(S)} ann_S(M)$.

The reverse inclusion follows by Theorem \ref{iff simple} and Lemma \ref{delta m}, similarly as in the proof of (i). Therefore $rad_s(S) = \cap_{\mu \in \mathcal{RC}_s(S)} \mu$.
\end{proof}
\begin{example}
Let $F$ be a semifield. Then $F$ has no nontrivial proper ideals. Hence $\Delta_F \in \mathcal{RC}_m(F)$ which implies that $rad_m(F)=\Delta_F$.

In particular, the max-plus algebra $\mathbb{R}_{max} = (\mathbb{R} \cup \{-\infty\}, max, + )$ is a semifield. Hence $rad_m(\mathbb{R}_{max}) = \Delta_{\mathbb{R}_{max}}$.
\end{example}

Subsemiring of an $m$-semisimple semiring may not be $m$-semisimple, as we see in the following example.
\begin{example}
Consider the subsemiring $S = \mathbb{R}^+ \cup\{0, -\infty \}$ of the max-plus algebra $\mathbb{R}_{max}$ where $\mathbb{R}^+$ is the set of all positive reals. Then $S$ is an additively idempotent semiring with $1_S = 0$. Hence every right congruence on $S$ is a regular congruence.

Let $\rho \neq \Delta_S, \nabla_S$ be a congruence on $S$. Since $\{-\infty\}$ and $S$ are the only two saturated ideals in $S$, we have $[-\infty]_{\rho} = \{-\infty\}$. Then there exist $x, y \in \mathbb{R}^+$ and $x < y$ such that $(x, y) \in \rho $. If $x < y$ and $(x, y) \in \rho$ then for $a=y-x$ we have $x \rho (y+na)$ for every positive integer $n$; and so $[x, \infty) \times [x, \infty) \subseteq \rho$. Hence we have $\rho = (J \times J) \cup \Delta_S$ where $J$ is either $[a,\infty)$ or $(a,\infty)$, by the completeness property of $\mathbb{R}$. If $J \neq [0, \infty)$ then $J \cup \{-\infty\}$ is a $\rho$-saturated proper ideal containing $[-\infty]_{\rho}$, which implies that $\rho \notin  \mathcal{RC}_m(S)$. Hence $\rho=([0, \infty) \times [0, \infty)) \cup \Delta_S$ is the only m-regular congruence on $S$, and  it follows that $rad_m(S) = ([0, \infty) \times [0, \infty)) \cup \Delta_S$.
\end{example}

If $\rho$ is a right congruence on $S$ then $S/\rho$ is a right S-semimodule such that $ann_S(S/\rho)=\{ (x,y) \in \nabla_S \mid \; (sx,sy) \in \rho \mbox{ for all } s \in S\}$. Define
\[
(\rho:\nabla_S)=\{ (x,y) \in \nabla_S \mid \; (sx,sy) \in \rho \mbox{ for all } s \in S\}.
\]
Then $(\rho:\nabla_S)$ is a congruence on $S$. Furthermore, if $\rho$ is regular, there exists an element $e \in S$ such that $(es,s) \in \rho$ for all $s \in S$. Let $(x,y) \in (\rho :\nabla_S)$. Then $(ex,ey) \in \rho$ implies that $x \rho ex \rho ey \rho y$ and so $(x, y) \in \rho$. Thus $(\rho: \nabla_S) \subseteq \rho$. In fact, $(\rho:\nabla_S)$ is the largest congruence on $S$ contained in $\rho$ for every regular right congruence $\rho$ on $S$.

Now $(\rho:\nabla_S)=ann_S(S/\rho)$ together with Theorem \ref{iff minimal} and Theorem \ref{iff simple} turn out to be yet another characterization of the m-radical and s-radical of a semiring.
\begin{theorem}
For every semiring $S$, we have
\begin{enumerate}
\item
$rad_m(S) = \cap_{\rho \in \mathcal{RC}_m(S)} (\rho: \nabla_S)$, and
\item
$rad_s(S) = \cap_{\rho \in \mathcal{RC}_s(S)} (\rho: \nabla_S)$.
\end{enumerate}
\end{theorem}

Let $R$ and $S$ be two semirings. Then $R \times S$ is a semiring where the addition and the multiplication are defined componentwise. Consider two right congruences $\sigma$ and $\eta$ on $R$ and $S$ respectively. Define
\begin{center}
$\sigma \times \eta = \{ ((r_1, s_1), (r_2, s_2)) \mid (r_1, r_2) \in \sigma \mbox{ and } (s_1, s_2) \in \eta \}$.
\end{center}

Then $\sigma \times \eta$ is a right congruence on the semiring $R \times S$ where $[(0, 0)]_{\sigma \times \eta} = \{(r, s) \in R \times S \mid (r, 0) \in \sigma \ \mbox{and} \ (0, s) \in \eta \} = [0]_{\sigma} \times [0]_{\eta}$. If moreover, $\sigma$ and $\eta$ are regular, then $\sigma \times \eta$ is also regular on $R \times S$.

For every right congruence $\rho$ on $R \times S$, define
\begin{align*}
& \rho_R = \{ (r_1, r_2) \in R \times R \mid \exists \; s_1, s_2 \in S \mbox{ such that } (r_1, s_1) \rho (r_2, s_2) \} \\
\mbox{and } & \rho_S = \{ (s_1, s_2) \in S \times S \mid \exists \ r_1, r_2 \in R \mbox{ such that } (r_1, s_1) \rho (r_2, s_2) \}.
\end{align*}

Then $\rho \subseteq \rho_R \times \rho_S$. The following result shows that the equality holds if $\rho$ is a $m$-regular congruence on $R \times S$.
\begin{lemma}                                               \label{Lemma1}
Let $R$ and $S$ be two semirings. Then $\rho \in \mathcal{RC}_m(R\times S)$ if and only if $\rho = \sigma \times \nabla_S$ or $\rho = \nabla_R \times \delta$ where $\sigma \in \mathcal{RC}_m(R)$ and $\delta \in \mathcal{RC}_m(S)$.
\end{lemma}
\begin{proof}
First assume that $\rho \in \mathcal{RC}_m(R\times S)$. Then both $\rho_R$ and $\rho_S$ are regular right congruences on $R$ and $S$, respectively. Also $[0]_{\rho_R}\times [0]_{\rho_S}$ is a proper right ideal of $R \times S$ such that $[(0,0)]_{\rho} \subseteq [0]_{\rho_R}\times [0]_{\rho_S}$.

Now consider $(r, s) \in [0]_{\rho_R}\times [0]_{\rho_S}$ and $(r', s') \in R \times S$ such that $(r', s') \rho (r,
s)$. Then $r' \rho_R r \rho_R 0$ and $s' \rho_S s \rho_S 0$ implies that $(r', s') \in [0]_{\rho_R}\times [0]_{\rho_S}$. Hence $[0]_{\rho_R}\times [0]_{\rho_S}$ is $\rho$-saturated right ideal of $R\times S$. Since $[(0,0)]_{\rho}$ is a maximal $\rho$-saturated right ideal of $R \times S$, it follows that $[(0,0)]_{\rho} = [0]_{\rho_R}\times [0]_{\rho_S}$. Therefore either $[0]_{\rho_R} \neq R$ or $[0]_{\rho_S} \neq S$.

Suppose that $[0]_{\rho_R} \neq R$. Let $I$ be a $\rho_R-saturated$ proper right ideal of $R$ such that $[0]_{\rho_R}
\subseteq I$. Then $I \times S$ is a proper $\rho-saturated$ right ideal of $R \times S$ such that $[0]_{\rho_R}\times [0]_{\rho_S} \subseteq I \times S$. Then maximality of $[(0,0)]_{\rho}$ implies that $[0]_{\rho_R}\times [0]_{\rho_S} = I \times S$ and hence $[0]_{\rho_R} = I$ and $[0]_{\rho_S} = S$. Therefore $[0]_{\rho_R}$ is a maximal $\rho_R-saturated$ right ideal of $R$, which implies that $\rho_R \in \mathcal{RC}_m(R)$. Also $[0]_{\rho_S}= S$ implies that $\rho_S = \nabla_S$.

Now $[(0,0)]_{\rho} = [0]_{\rho_R} \times S$ implies that $(0, s_1) \rho (0, s_2)$ and so $(r, s_1) \rho (r,
s_2)$ for all $r \in R$, $s_1, s_2 \in S$. Let $(r_1, s_1) \rho_R \times \nabla_S (r_2, s_2)$. Then $r_1 \rho_R
 r_2$ which implies that $(r_1, s) \rho (r_2 , s')$ for some $s, s' \in S$. Hence $(r_1,s_1) \rho (r_1, s) \rho
(r_2, s') \rho (r_2, s_2)$ and so $\rho_R \times \nabla_S \subseteq \rho$. Also $\rho \subseteq \rho_R \times \nabla_S$. Therefore $\rho = \rho_R \times \nabla_S$.

If $[0]_{\rho_S} \neq S$, then similarly it follows that $\rho = \nabla_R \times \rho_S$.

Conversely let $\sigma \in \mathcal{RC}_m(R)$. Then $\sigma \times \nabla_S$ is a regular right congruence on $R \times S$ with  $[(0, 0)]_{\sigma \times \nabla_S} = [0]_{\sigma} \times S$. Let $J$ be a proper $\sigma \times \nabla_S$-saturated ideal in $R \times S$ containing $[0]_{\sigma \times \nabla_S}$. Denote $J_R = \{ r \in R \mid \exists s \in S \mbox{ such that } (r, s) \in J \}$. Then $J=J_R \times S$ and $J_R$ is a proper $\sigma$-saturated right ideal in $R$ containing $[0]_{\sigma}$. Therefore $J_R = [0]_{\sigma}$ and so $J = [0]_{\sigma} \times S$. Hence $\sigma \times \nabla_S \in \mathcal{RC}_m(R \times S)$. Similarly it follows that $\nabla_R \times \eta \in \mathcal{RC}_m(R \times S)$ for every $\eta \in \mathcal{RC}_m(S)$.
\end{proof}

The following theorem characterizes the Jacobson m-radical of the product semiring $R \times S$ in terms of the Jacobson m-radicals of the component semirings $R$ and $S$.
\begin{theorem}
Let $R$ and $S$ be two semirings. Then $rad_m(R \times S) = rad_m(R) \times rad_m(S)$.
\end{theorem}
\begin{proof}
We have,
\begin{align*}
rad_m(R \times S) &= \cap_{\rho \in  \mathcal{RC}_m(R \times S)} \rho \\
                  &= (\cap_{\rho_R \in \mathcal{RC}_m(R)}(\rho_R \times \nabla_S)) \cap (\cap_{\rho_S \in \mathcal{RC}_m(S)}(\nabla_R \times \rho_S)) \\
                  &= ((\cap_{\rho_R \in \mathcal{RC}_m(R)}\rho_R) \times \nabla_S) \cap (\nabla_R \times (\cap_{\rho_S \in \mathcal{RC}_m(S)} \rho_S)) \\
                  &= (rad_m(R) \times \nabla_S) \cap (\nabla_R \times rad_m(S)) \\
                  &= rad_m(R) \times rad_m(S).
\end{align*}
\end{proof}

Similarly, it can be proved analogous lemmas related to s-radical to get the following result whose proof is omitted.
\begin{theorem}
Let $R$ and $S$ be two semirings. Then $rad_s(R \times S) = rad_s(R) \times rad_s(S)$.
\end{theorem}

\section{Jacobson semisimple semirings}
In this section, we study the structure of Jacobson m-semisimple and s-semisimple semirings.

A semiring $S$ is Jacobson m-semisimple if $\cap_{M \in \mathcal{RC}_m(S)}ann_S(M) = \Delta_S$. Hence every m-semisimple semiring is a subdirect product of the family of semirings $\{S/ann_S(M) \mid M \in \mathcal{RC}_m(S)\}$. Also, it follows from Lemma \ref{Lemma 2.1} that if $M$ is an $S$-semimodule, then $M$ is a faithful $S/ann_S(M)$-semimodule. If moreover, $M$ is a minimal $S$-semimodule, then $M$ is so as an $S/ann_S(M)$-semimodule. Similarly, every $s$-semisimple semiring $S$ is a subdirect product of the family of semirings $\{S/ann_S(M) \mid M \in \mathcal{RC}_s(S)\}$ where each quotient semiring $S/ann_S(M)$ has the property that $M$ is faithful and a simple semimodule over the quotient semiring $S/ann_S(M)$. Intending to characterize the structure of semisimple semirings, we introduce the following two notions.

\begin{definition}
Let $S$ be a semiring. Then $S$ is called
\begin{enumerate}
\item[(i)] $m$-primitive if there is a faithful minimal $S$-semimodule $M$;
\item[(ii)] $s$-primitive if there is a faithful simple $S$-semimodule $M$.
\end{enumerate}
\end{definition}

If $S$ is an $m$-primitive semiring, then there is a minimal $S$-semimodule $M$ such that $ann_S(M) = \Delta_S$. Hence $rad_m(S) = \cap_{M \in \mathcal{M}(S)} ann_S(M) = \Delta_S$ and so $S$ is $m$-semisimple. Similarly, every $s$-primitive semiring is $s$-semisimple.

A congruence $\sigma$ on $S$ is said to be an \emph{$m$-primitive ($s$-primitive) congruence} if the quotient semiring $S/\sigma$ is an $m$-primitive ($s$-primitive) semiring. Thus $\sigma$ is $m$-primitive($s$-primitive) if and only if there exists a faithful minimal(simple) $S/\sigma$-semimodule $M$.
\begin{lemma}        \label{primitive congruence iff}
Let $\sigma$ be a congruence on $S$. Then the following conditions are equivalent:
\begin{enumerate}
\item
$\sigma$ is $m$-primitive (s-primitive);
\item
$\sigma = ann_S(M)$ for some minimal (simple) $S$-semimodule $M$;
\item
$\sigma = (\rho: \nabla_S)$ for some $\rho \in \mathcal{RC}_m(S)$ ($\rho \in \mathcal{RC}_s(S)$).
\end{enumerate}
\end{lemma}
\begin{proof}
we prove the result for $m$-primitive congruences. The other cases are similar.
\\ $(i) \Rightarrow (ii):$ Let $\sigma$ be an $m$-primitive congruence on $S$. Then there exists a faithful minimal $S/\sigma$-semimodule $M$. Hence, by Lemma \ref{Lemma 2.1}, $M$ is also a minimal $S$-semimodule such that $\sigma \subseteq ann_S(M)$ and $\Delta_{S/\sigma} = ann_{S/\sigma}(M) = ann_S(M)/\sigma$, i.e., $\sigma = ann_S(M)$.
\\ $(ii) \Rightarrow (iii):$ Let $M$ be a minimal $S$-semimodule and $\sigma = ann_S(M)$. Then $M$ is a minimal and faithful right $S/\sigma$-semimodule. Theorem \ref{iff minimal} implies that there exists $\rho \in \mathcal{RC}_m(S)$ such that $M \simeq S/\rho$; and hence $ann_S(M) = (\rho : \nabla_S)$. Then $ann_{(S/\sigma)}(M) = ann_S(M)/\sigma = \Delta_{(S/\sigma)}$ implies that $ann_S(M)=\sigma$. Hence $(\rho : \nabla_S) = ann_S(M) = \sigma$.
\\ $(iii) \Rightarrow (i):$ Let $\rho \in \mathcal{RC}_m(S)$ and $\sigma = (\rho: \nabla_S)$. Then $S/\rho$ is a minimal right $S$-semimodule and $ann_S(S/\rho) = (\rho: \nabla_S) = \sigma$. Since $\sigma$ is a semiring congruence on $S$ and $\sigma = (\rho: \nabla_S)$, it follows that $S/\rho$ is a minimal right $S/\sigma$-semimodule. Also, by Lemma \ref{Lemma 2.1}, we have $ann_{S/\sigma}(S/\rho) = ann_S(S/\rho)/\sigma$ = $\Delta_{S/\sigma}$. Hence $S/\rho$ is a minimal and faithful right $S/\sigma$-semimodule, so $\sigma$ is a $m$-primitive congruence on $S$.
\end{proof}

From the definition, it follows that a semiring $S$ is an $m$-primitive ($s$-primitive) semiring if and only if $\Delta_S$ is an $m$-primitive ($s$-primitive) congruence on $S$. Thus we have:
\begin{corollary}    \label{primitive iff}
Let $S$ be a semiring. Then $S$ is
\begin{enumerate}
\item
$m$-primitive if and only if there exists $\rho \in \mathcal{RC}_m(S)$ such that $(\rho: \nabla_S)= \Delta_S$;
\item
$s$-primitive if and only if there exists $\rho \in \mathcal{RC}_s(S)$ such that $(\rho: \nabla_S)= \Delta_S$.
\end{enumerate}
\end{corollary}

Division semiring is a noncommutative generalization of a semifield. The following result shows that primitive semirings are other noncommutative generalizations of semifields. The m-primitive semirings generalize the semifields, whereas the s-primitive semirings generalize the congruence-simple semifields.
\begin{theorem}                        \label{m-primitive iff semifield}
Let $S$ be a commutative semiring. Then $S$ is
\begin{enumerate}
\item
$m$-primitive if and only if it is a semifield.
\item
$s$-primitive if and only if it is a congruence-simple semifield.
\end{enumerate}
\end{theorem}
\begin{proof}
(i) \ Let $S$ be a commutative $m$-primitive semiring. Then, by Corollary \ref{primitive iff}, there is a regular right congruence $\rho$ in $\mathcal{RC}_m(S)$ such that $( \rho : \nabla_S ) = \Delta_S$. Since $S$ is a commutative semiring, $\rho$ becomes a congruence on $S$ and so $\rho = ( \rho : \nabla_S ) = \Delta_S$. Therefore $\Delta_S \in \mathcal{RC}_m(S)$ and there is an element $e \in S$ such that $es = s = se$ for all $s \in S$. Thus $e$ is a multiplicative identity in $S$. Also $\rho = \Delta_S \in \mathcal{RC}_m(S)$ implies that $(0)$ is maximal $\Delta_S$-saturated ideal in $S$. Since every ideal in $S$ is $\Delta_S$-saturated, it follows that $(0)$ and $S$ are the only two ideals in $S$. Now for each non-zero element $a \in S$, $aS$ is a non-zero ideal in $S$. Hence $aS = S$, which implies that there exists an element $b \in S$ such that $ab = e = ba$. Thus $S$ is a semifield.

 Conversely, let $S$ be a semifield. Then $M=S$ is a minimal $S$-semimodule and $ann_S(M) = \{ (s_1, s_2) \in S \times S \mid ss_1 = ss_2 \; \textrm{for all} \; s \in S \} = \Delta_S$. Therefore $S$ is $m$-primitive.
\\ (ii) \ Let $S$ be a commutative $s$-primitive semiring. Then $S$ is $m$-primitive, and so, by (i), it is a semifield. Also, by Corollary \ref{primitive iff}, there exists a right congruence $\rho \in \mathcal{RC}_s(S)$ such that $(\rho : \nabla_S)=\Delta_S$. Since $S$ is commutative, it follows that $(\rho : \nabla_S)=\rho$. Hence $\Delta_S=\rho \in \mathcal{RC}_s(S)$ which implies that $M=S/\rho \simeq S$ is a simple $S$-semimodule. Hence the semifield $S$ is congruence-simple.

Conversely, if $S$ is a congruence-simple semifield, then $S$ itself is a faithful simple $S$-semimodule. Hence $S$ is $s$-primitive.
\end{proof}

Theorem \ref{m-primitive iff semifield} tells us that the congruence-simple semifields constitute an important subclass of the semifields. Similarly to the fields, the Krull-dimension of a congruence-simple semifield is 0, whereas there are semifields, say, for example, $\mathbb{R}_{max}$ having the Krull-dimension 1 \cite{JM2018JA}. A semiring $S$ is called \emph{zerosumfree} if for every $a, b \in S$, we have $a+b=0$ implies that $a=0$ and $b=0$. It is well known that a semifield $S$ is either zerosumfree or is a field [Proposition 4.34; \cite{Golan1999}]. Every field is a congruence-simple semifield. If $S$ is zerosumfree, then $\rho=\{(s,t) \in S \times S \mid s \neq 0 \neq t\} \cup \{(0, 0)\}$ is a congruence on $S$. So for $S$ to be congruence-simple, we must have $|S|=2$. Then $S$ is the 2-element Boolean algebra $\mathbb{B}$. Thus a congruence-simple semifield is either the 2-element Boolean algebra $\mathbb{B}$ or a field.

However, in the following, we include independent proof.
\begin{theorem}                    \label{congruence-simple iff field}
Let $S$ be a semiring with $|S| > 2$. Then $S$ is a congruence-simple semifield if and only if it is a field.
\end{theorem}
\begin{proof}
First, assume that $S$ is a congruence-simple semifield. Denote $Z(S)=\{x \in S \mid x+y=0 \ \mbox{for some} \ y \in S\}$. Then $Z(S)$ is an ideal of $S$; and so $Z(S)$ is either $\{0\}$ or $S$. If $Z(S)=\{0\}$, then $S$ is zerosumfree. So $\{(0,0)\} \cup \{(s, t) \in S \times S \mid s\neq 0\neq t\}$ induces a nontrivial congruence on $S$, which contradicts that $S$ is congruence-simple. Hence $Z(S)=S$ which implies that $(S, +)$ is a group. Thus $S$ is a field.

Converse follows trivially.
\end{proof}

Thus a zerosumfree semifield $S$ with $|S| > 2$ can not be congruence-simple. So, in particular, the max-plus algebra $\mbr_{max}$ is a semifield but not congruence-simple. Hence $\mbr_{max}$ is m-primitive but not s-primitive.

The 2-element Boolean algebra $\mathbb{B}$ and the field $\mbz_2$ of all integers modulo 2 are the only semifields of order two up to isomorphism. Hence it turns out to be the following specific characterization of the commutative s-primitive semirings.
\begin{corollary}                            \label{commutative s-primitive semirings}
A commutative semiring $S$ is s-primitive if and only if it is either the 2-element Boolean algebra $\mathbb{B}$ or a field.
\end{corollary}

A semiring $S$ is called a \emph{subdirect product} of a family $\{S_{\alpha}\}_{\Delta}$ of semirings if there is an one-to-one semiring homomorphism $\phi : S \longrightarrow \prod_{\Delta}S_{\alpha}$ such that for each $\alpha \in \Delta$, the composition $\pi_{\alpha} \circ \phi : S \longrightarrow S_{\alpha}$ is onto where $\pi_{\alpha} : \prod_{\Delta}S_{\alpha} \longrightarrow S_{\alpha}$ is the projection mapping.

It is well known that a semiring $S$ is a subdirect product of a family $\{S_{\alpha}\}_{\Delta}$ of semirings if and only if there is a family $\{\rho_{\alpha}\}_{\Delta}$ of congruences on $S$ such that $S/\rho_{\alpha} \simeq S_{\alpha}$ for every $\alpha \in \Delta$ and $\cap_{\Delta} \rho_{\alpha} = \Delta_S$.
\begin{theorem}                                    \label{structure of semisimple semirings}
A semiring $S$ is $m$-semisimple ($s$- semisimple) if and only if it is a subdirect product of $m$-primitive ($s$-primitive) semirings.
\end{theorem}
\begin{proof}
We prove the result for $m$-semisimple semirings. Proof for $s$-semisimple semirings is similar.

First, assume that $S$ is a $m$-semisimple semiring. Then $rad_m(S) = \cap_{M \in
\mathcal{M}(S)} ann_S(M) = \Delta_S$. Hence $S$ is a subdirect product of the family $\{ S/ann_S(M) \mid M \in \mathcal{M}(S) \}$ of semirings. Lemma \ref{primitive congruence iff} implies that $ann_S(M)$ is an $m$-primitive congruence on $S$ for every minimal $S$-semimodule $M$. Therefore every semiring in the family $\{ S/ann_S(M) \mid M \in \mathcal{M}(S) \}$ is an $m$-primitive semiring, and so $S$ is a subdirect product of m-primitive semirings.

Conversely, let $S$ be a subdirect product of a family of $m$-primitive semirings $\{S_i \mid \textrm{for all} \; i \in \Lambda\}$. Then there exists a one-to-one homomorphism $\phi : S \rightarrow \Pi_{i \in \Lambda} S_i$ such that the mapping $\pi_i \circ \phi : S \longrightarrow S_i$ is onto for all $i \in \Lambda$. Thus $S/ker (\pi_i \circ \phi) \cong S_i$ for all $i \in \Lambda$. Let $M_i$ be a faithful minimal $S_i$-semimodule for each $i \in \Lambda$. Then, by the Lemma \ref{Lemma 2.1}, $M_i$ is a minimal $S$-semimodule where $ms=m\pi_i \circ \phi(s)$ for all $s \in S$ and $m \in M_i$. Hence $\cap_{M \in \mathcal{M}(S)} ann_S(M) \subseteq \cap_{i \in \Lambda} ann_S(M_i)$. Now $(a, b) \in ann_S(M_i)$ implies that $m \pi_i \circ \phi(a) =m \pi_i \circ \phi(b)$ for all $m \in M_i$; and so $(\pi_i \circ \phi(a), \pi_i \circ \phi(b)) \in ann_{S_i}(M)$. Since $M_i$ is faithful over $S_i$, it follows that $\pi_i \circ \phi(a)=\pi_i \circ \phi(b)$. Hence $\cap_{i \in \Lambda} ann_S(M_i)=\Delta_S$ which implies that $rad_m(S) = \cap_{M \in \mathcal{M}(S)} ann_S(M) =\Delta_S$. Thus $S$ is a $m$-semisimple semiring.
\end{proof}

Now, taken together the structure of an s-semisimple semiring characterized in Theorem \ref{structure of semisimple semirings} and the characterization of the commutative s-primitive semirings in Corollary \ref{commutative s-primitive semirings} turn out to be an characterization of the commutative s-semisimple semirings.
\begin{corollary}
Let $S$ be a commutative semiring. Then $S$ is an $s$-semisimple semiring if and only if it is a subdirect product of a family of semirings that are either the 2-element Boolean algebra $\mathbb{B}$ or fields.
\end{corollary}

Mischell and Fenoglio \cite{MF} and Basir et al. \cite{BHJK} independently proved that a commutative semiring $S$ with $|S| \geqslant 2$ is congruence-simple if and only if it is either a field or the 2-element Boolean algebra $\mathbb{B}$. Hence it follows that a commutative semiring is s-semisimple if and only if it is a subdirect product of congruence-simple commutative semirings. A semiring homomorphism $f : S_1 \longrightarrow S_2$ is said to be \emph{semiisomorphism} if, for every $a \in S_1$, we have $f(a)=0$ only for $a=0$. Katsov and Nam \cite{KN2014} proved that a commutative semiring $S$ is Brown-McCoy semisimple if and only if $S$ is semi-isomorphic to a subdirect product of a family of semirings that are either the 2-element Boolean algebra $\mathbb{B}$ or fields. Hence every commutative s-semisimple semiring is Brown-McCoy semisimple in the sense of Katsov and Nam.
\begin{example}
Consider the semiring $\mathbb{N}$ of all nonnegative integers. Then for every prime $p$, the Bourne congruence $\sigma_{p\mathbb{N}}$ is a maximal regular congruence on $\mathbb{N}$ with $[0]_{\sigma_{p\mathbb{N}}}=p\mathbb{N}$. If $J$ is a $\sigma_{p\mathbb{N}}$-saturated ideal in $\mathbb{N}$ with $p\mathbb{N} \varsubsetneq J$, then there exists $a \in J$ such that $0 < a < p$. By the Fermat's little theorem, we have $a^{p-1} \equiv 1 (mod p)$ which implies that $1 \in J$ and so $J = \mathbb{N}$. Thus $p\mathbb{N} = [0]_{\sigma_{p\mathbb{N}}}$ is a maximal $\sigma_{p\mathbb{N}}$-saturated ideal in $\mathbb{N}$ and it follows that $\sigma_{p\mathbb{N}} \in \mathcal{RC}_s(\mathbb{N})$. Hence $rad_s(\mathbb{N}) \subseteq \cap \sigma_{p\mathbb{N}} = \Delta_{\mathbb{N}}$; and so $\mathbb{N}$ is an $s$-semisimple semiring.

Also $\cap \sigma_{p\mathbb{N}} = \Delta_{\mathbb{N}}$ implies that $\mathbb{N}$ is a subdirect product of the family of fields $\mathbb{N}_p = \mathbb{N}/\sigma_{p\mathbb{N}}$ where $p$ is a prime.
\end{example}

We conclude this section with a representation of s-primitive semirings as a semiring of endomorphisms on a semimodule over a division semiring.

The opposite semiring $S^{op}$ of a semiring $(S, +, \cdot)$ is defined by $(S, +, \circ)$ where $a \circ b = b \cdot a$ for all $a, b \in S$. Hence a semiring $S$ is a division semiring if and only if the opposite semiring $S^{op}$ is so.

Let $M$ be a semimodule over a division semiring $D$. Then a subsemiring $T$ of the endomorphism semiring $End_D(M)$ is called \emph{1-fold transitive} if for every non-zero $m \in M$ and $n \in M$ there exists $\alpha \in T$ such that $\alpha(m) = n$.

In the context of semirings, Schur's lemma was proved in \cite{IRS2011}, which states that if $M$ is a simple $S$-semimodule, then the endomorphism semiring $End_S(M)$ is a division semiring.

Let $M$ be a right $S$-semimodule and $E=End_S(M)$. Then for $D=E^{op}$, $M$ is a right semimodule over $D$ where the scalar multiplication is defined by $m \cdot \alpha = \alpha(m)$ for all $m \in M$ and $\alpha \in D$.
\begin{theorem}                       \label{s-primitive as 1-fold transitive}
If $S$ is a right $s$-primitive semiring, then $S^{op}$ is isomorphic to a 1-fold transitive subsemiring of the semiring $End_D(M)$ of all endomorphisms on a semimodule $M$ over a division semiring $D$.
\end{theorem}
\begin{proof}
Let $M$ be a faithful simple right $S$-semimodule. By Schur's Lemma for semimodules \cite{IRS2011}, the semiring $E = End_S(M)$ is a division semiring. Hence $D= E^{op}$ is a division semiring, and so $M$ as a right $D$-semimodule where $m \cdot \alpha \mapsto \alpha(m)$.

For every $a \in S$, define a mapping $\psi_a : M \rightarrow M$ by $\psi_a (m) =ma$.
Then for every $\alpha \in D$, we have $\psi_a(m.\alpha) = \psi_a (\alpha(m))= \alpha(m)a = \alpha(ma) = (ma).\alpha = \psi_a(m). \alpha$. In fact, $\psi_a$ is an endomorphism on $M$ considered a $D$-semimodule.

Also the mapping $\psi : S^{op} \rightarrow End_D(M)$ defined by $\psi(a) = \psi_a$ is a semiring homomorphism. Moreover $ker \ \psi = ann_S(M) = \Delta_S$ implies that $\psi$ is an injective homomorphism; and so $S^{op}$ is isomorphic to the subsemiring $T = \{ \psi_a \mid a \in S \}$ of $End_D(M)$.

Since $M$ is a simple right $S$-semimodule, by Lemma \ref{minimal}, for every $m(\neq 0) \in M$, $mS = M$.
Then for every $n \in M$ there exists $a \in S$ such that $ma = n$ and so $\psi_a(m) = n$. Thus $T$ is a 1-fold transitive subsemiring of $End_D(M)$.
\end{proof}

It follows from Corollary \ref{commutative s-primitive semirings} that the semifield $F = \mathbb{R}_{max}$ is not an $s$-primitive semiring. Since $F$ contains $1$, every $F$-endomorphism on $F$ is of the form $\psi_a : F \rightarrow F$ given by $\psi_a(m) = am$. Hence $F \simeq End_F(F)$ which implies that $End_F(F)$ is not $s$-primitive; whereas $End_F(F)$ is a 1-fold transitive subsemiring of itself. Thus the converse of the Theorem \ref{s-primitive as 1-fold transitive} does not hold. However, the converse holds in the following weaker form.
\begin{theorem}
Let $D$ be a division semiring and $M$ be a right $D$-semimodule. If $T$ is a 1-fold transitive subsemiring of $End_D(M)$, then $T^{op}$ is a right $m$-primitive semiring.
\end{theorem}
\begin{proof}
Define $M \times T^{op} \rightarrow M$ by $m.\alpha \mapsto \alpha(m)$. Then $M$ is a right $T^{op}$-semimodule. Let $m$ be a non-zero element in $M$. Then for every $n \in M$, there exists $\alpha \in T$ such that $m.\alpha = n$. Therefore $mT^{op} = M$ which implies that $M$ is minimal, by Lemma \ref{minimal}.
Now \begin{align*}ann_{T^{op}}(M) &= \{ (\alpha, \beta) \in T \times T \mid m.\alpha = m.\beta \mbox{ for all } m \in M \}\\
&= \{ (\alpha, \beta) \in T \times T \mid \alpha(m) = \beta(m) \mbox{ for all } m \in M \}\\
&= \{ (\alpha, \beta) \in T \times T \mid \alpha = \beta \}\\
&= \Delta_S
\end{align*} and so $M$ is a faithful minimal $T^{op}$-semimodule. Therefore $T^{op}$ is a $m$-primitive semiring.
\end{proof}

\section{Remarks and open questions}
In Section 3, we introduced the m-radical and s-radical of a semiring based on the two notions of `simplicity' of a semimodule, namely minimal semimodule and simple semimodule. Similarly, the e-radical of a semiring can also be defined based on the class of congruence simple semimodules, which are known as elementary semimodules \cite{Chen}, \cite{IRS2011}. We conjecture that an $S$-semimodule $M$ is elementary if and only if there exists a maximal regular right congruence $\mu$ on $S$ such that $S/\mu$ is isomorphic to $M$. Once this conjecture is proved, it would be possible to characterize the e-radical of a semiring internally. Similarly to the present work, every e-semisimple semiring can be expressed as a subdirect product of e-primitive semirings. Thus the present work can be extended to characterize the e-radical of a semiring and e-semisimple semirings.

\bibliographystyle{amsplain}

\begin{thebibliography}{10}
\baselineskip 5mm

\bibitem{ABG}
M. Akian, R. Bapat, S. Gaubert, Max-plus algebra, in: L. Hogben, R. Brualdi, A. Greenbaum, R. Mathias (Eds.), Handbook of Linear Algebra, Chapman $\&$ Hall, London, 2006.

\bibitem{BHJK}
R. El Basir, J Hurt, A. Jan\v{c}a\v{r}ik, T. Kepka, Simple commutative semirings, J. Algebra 236 (2001) 277-306.

\bibitem{BE2017}
A. Bertram, R. Easton, The tropical Nullstellensatz for congruences, Adv. Math. 308 (2017) 36-82.

\bibitem{BM2010}
A. K. Bhuniya, T. Mondal, Distributive lattice decompositions of semirings with a semilattice additive reduct 80 (2010) 293 - 301.

\bibitem{BM2015}
On the least distributive lattice congruence on a semiring with a semilattice additive reduct, Acta math. Hunger. 147 (2015) 189 - 204.

\bibitem{Bourne1951}
S. Bourne, The Jacobson radical of a semiring, Proc. Nat. Acad. Sci. U.S.A. 37 (1951) 163-170.

\bibitem{Castella2010}
D. Castella, El\'{e}ments \'{d}alg\`{e}bre lin\'{e}aire tropicale, Linear Algebra and its Applications 432 (2010) 1460-1474.

\bibitem{Chen}
C. Chen et. al., Extreme representations of semirings, Serdica Mathematical Journal 44(3/4)(2018) 365-412.

\bibitem{Conway}
J. H. Conway, Regular Algebra and Finite Machines, Chapman $\&$ Hall, London, 1971.

\bibitem{CC2009}
A. Connes, C. Consani, Characteristic one, entropy and the absolute point, preprint; arXiv:0911.3537; Proceedings of the JAMI Conference 2009.

\bibitem{CC2010}
A. Connes, C. Consani, Schemes over F1 and zeta functions, Compositio Mathematica 146 (6) (2010) 1383-1415.

\bibitem{CCM}
A. Connes, C. Consani, M. Marcolli, Fun with F1, J. Number Theory 129 (6) (2009) 1532-1561.

\bibitem{Deitmar2008}
A. Deitmar, F1-schemes and toric varieties, Beitr\"{a}ge Algebra Geom. 49 (2) (2008) 517-525.

\bibitem{Gathmann}
A. Gathmann, Tropical algebraic geometry, Jahresber. Deutsch. Math. -Verein. 108 (1) (2006) 3-32.

\bibitem{Golan1999}
J. S. Golan, Semirings and Their Applications, Dordrecht-Boston-London:
Kluwer Academic Publishers (1999).

\bibitem{HW}
U. Hebisch, H. J. Weinert, Semirings : Algebraic Theory and Applications in Computer Science, World Scientific (1993).

\bibitem{HW1997}
U. Hebisch, H. J. Weinert, Radical theory for semirings, Quaestiones
Mathematicae 20 (1997) 647-661.

\bibitem{HW2001}
U. Hebisch, H. J. Weinert, On the interrelation between radical theories for
semirings and rings, Comm. Algebra 29 (2001) 109-129.

\bibitem{Herstein}
I. N. Herstein, Noncommutative Rings, The Carus Mathematical Monographs : The Mathematical Association of America (1971).

\bibitem{Hoehnke1964}
H. J. Hoehnke, Structure of semigroups, Canad. J. Math. 18(1966) 449-491.

\bibitem{Il'in2021}
S, N, Il'in, On the homological classification of semirings, Journal of Mathematical Sciences 256(2) (2021) 125 - 143.

\bibitem{Ilzuka1959}
K. Ilzuka, On the Jacobson radical of a semiring, Tohuku Math. J. 11(1959) 409-421.

\bibitem{IMS}
I. Itenberg, G. Mikhalkin, E. Shustin, Tropical Algebraic Geometry, second ed., Oberwolfach Semin., vol. 35, Birkh\"{a}user Verlag, Basel, 2009.

\bibitem{IR2010}
Z. Izhakian, L. Rowen, Supertropical algebra, Adv. Math. 225 (4) (2010) 2222-2286.

\bibitem{IRS2011}
Z. Izhakian, J. Rhodes, B. Steinberg, Representation theory of finite semigroups over semirings, J. Algebra 336(2011) 139-157.

\bibitem{JM2018JA}
D. Jo\'{o}, K. Mincheva, On the dimension of polynomial semirings, J. Algebra 507(2018) 103-119.

\bibitem{JM2018SM}
D. Jo\'{o}, K. Mincheva, Prime congruences of additively idempotent semirings and a Nullstellensatz for tropical polynomials, Sel. Math. New Ser. 24(2018) 2207-2233.

\bibitem{KN2014}
Y. Katsov, T. G. Nam, On radicals of semirings and related problems, Comm. Algebra 42 (2014) 5065-5099.

\bibitem{Latorre1967}
D. R. LaTorre, A note on the Jacobson radical of a hemiring, Publ. Math. Debrecen 14 (1967) 9-13.

\bibitem{Lescot-II}
P. Lescot, Absolute algebra II - Ideals and spectra, J. Pure Appl. Algebra 215(7)(2011) 1782 - 1790.

\bibitem{Litviniv05-07}
G. L. Litvinov, The Maslov dequantization, and idempotent and tropical mathematics: a brief introduction, Zap. Nauchn. Sem. S. -Peterburg. Otdel. Mat. Inst. Steklov. (POMI) 326 (2005) 145-182, 282 (Teor. Predst. Din. Sist. Komb. i Algoritm. Metody. 13); translation in: J. Math. Sci. (N. Y.) 140 (3) (2007) 426-444.

\bibitem{MT}
L. H. Mai and N. X. Tuyen, Some remarks on the Jacobson radical types of semirings and related problems, Vietnam J. Math. 45 (2017) 493 - 506.

\bibitem{MF}
S. S. Mitchell, P. B. fenoglio, Congruence-free commutative semirings, Semigroup Forum 31 (1988) 79-91.

\bibitem{Morak}
B. Morak, On the radical theory for semirings, Beitrage zur Algebra und Geometrie 40 (1999) 533-549.

\bibitem{OJ1983}
D. M. Olson, T. L. Jenkins, Radical theory for hemirings, J. Natur. Sci. Math. 23 (1983) 23-32.

\bibitem{ON1989}
D. M. Olson, A. C. Nance, A note on radical for hemirings, Quaestiones Mathematicae 12 (1989) 307-314.

\bibitem{Schutzenberger}
M.P. Sch\"{u}tzenberger, On the definition of a family of automata, Inform. Control 4 (1961) 245-270.

\bibitem{SB2011}
M. K. Sen and A. K. Bhuniya, On semirings whose additive reduct is a semilattice, Semigroup Forum 82 (2011) 131 - 140.

\bibitem{WW1992}
H. J. Weinert, R. Weigandt, A Kurosh-Amitsur radical theory for proper semifields, Comm. Algebra 20 (1992) 2419-2458.

\bibitem{WW2003a}
H. J. Weinert, R. Weigandt, A new Kurosh-Amitsur radical theory for proper semifields I, Mathematica Pannonica 14 (2003) 3-28.

\bibitem{WW2003b}
H. J. Weinert, R. Weigandt, A new Kurosh-Amitsur radical theory for proper semifields II, Mathematica Pannonica 14 (2003) 149-164.

\bibitem{WJK}
D. Wilding, M. Johnson, M. Kambites, Exact rings and semirings, Journal of Algebra 388 (2013) 324 - 337.

\end{thebibliography}

\end{document}